\documentclass{amsart}
\usepackage{amssymb}
\usepackage[utf8]{inputenc}
\usepackage{xcolor}
\usepackage{pgf,tikz,pgfplots}
\pgfplotsset{compat=1.16}
\usepackage{float}
\usepackage{graphicx}
\usetikzlibrary{arrows}
\usepackage{enumitem}
\usepackage[colorlinks]{hyperref}

\usepackage{amsthm}
\theoremstyle{definition}
\newtheorem{theorem}{Theorem}[section]
\newtheorem{proposition}[theorem]{Proposition}

\newtheorem{lemma}[theorem]{Lemma}
\newtheorem*{definition}{Definition}
\newtheorem{remark}[theorem]{Remark}
\newtheorem{example}[theorem]{Example}
\newtheorem{conjecture}[theorem]{Conjecture}
\newtheorem{problem}[theorem]{Problem}
\newtheorem*{theorem_no_number}{Theorem}

\title{Jet Graphs}
\date{\today}

\author{Federico Galetto}
\address{Department of Mathematics and Statistics, Cleveland State University, Cleveland, OH, 44115-2215, USA}
\email{\href{mailto:f.galetto@csuohio.edu}{\nolinkurl{f.galetto@csuohio.edu}}}
\urladdr{\href{https://math.galetto.org}{\nolinkurl{https://math.galetto.org}}}

\author{Elisabeth Helmick}
\address{Department of Mathematics and Statistics, Cleveland State University, Cleveland, OH, 44115-2215, USA}
\email{\href{mailto:e.helmick@vikes.csuohio.edu}{\nolinkurl{e.helmick@vikes.csuohio.edu}}}

\author{Molly Walsh}
\address{Department of Mathematics and Statistics, Cleveland State University, Cleveland, OH, 44115-2215, USA}
\email{\href{mailto:m.m.walsh79@vikes.csuohio.edu}{\nolinkurl{m.m.walsh79@vikes.csuohio.edu}}}



\begin{document}

\begin{abstract}
  We define an operation of jets on graphs inspired by the
  corresponding notion in commutative algebra and algebraic
  geometry. We examine a few graph theoretic properties and invariants
  of this construction, including chromatic numbers, co-chordality,
  and vertex covers.
\end{abstract}

\maketitle
\section{Introduction}
\label{sec:introduction}

A differentiable function is one whose graph admits a tangent line at
any given point. Calculus students learn that differentiable functions
have good properties, for example, their graphs are smooth, presenting
no kinks, cusps, or sharp corners. By contrast, a function that is not
differentiable can have any of the aforementioned singularities.
These simple notions from calculus extend to more complicated
situations. An affine variety is the set of solutions of a system of
polynomial equations. Affine varieties, which are fundamental objects
in algebraic geometry, are smooth if they admit tangent planes (of the
appropriate dimension) at all their points. The collection of all the
tangent planes to an affine variety is known as the tangent scheme of
the variety. Its study can yield useful information on the
singularities of the original variety. In fact, the tangent scheme is
simply the first in a family of objects, called jet schemes, which are
important in singularity theory. The study of singularities via jets
was initiated by J.~Nash \cite{MR1381967}, and has generated
significant attention due to its connections with other topics of
interest in geometry, such as motivic integration and birational
geometry \cite{MR1905328,MR1856396,MR2483946}.

Jet schemes have been studied primarily from a geometric perspective.
For example, some progress has been made towards understanding jets of
determinantal varieties of generic and symmetric matrices thanks to
works such as
\cite{MR2389248,MR2100311,MR2166800,MR3270176,MR3990994}. The articles
just cited focus on describing irreducible components of jets of
determinantal varieties, along with certain algebraic invariants, such
as the Hilbert series. Another interesting article by R.A.~Goward and
K.E.~Smith \cite{MR2229478} describes the equations for the
reduced structure of the jet schemes of an arbitrary monomial scheme,
showing in particular that it is itself a monomial scheme.

The present article uses the work of Goward and Smith as a starting
point to define an operation of jets for graphs that is compatible
with the existing notion of jets in algebra and geometry.  The
construction proceeds as follows.
\begin{enumerate}
\item Consider a simple undirected graph $G$.
\item\label{item:1} Construct the edge ideal $I(G)$ of $G$ in a
  suitable polynomial ring (see \cite[Definition 3]{MR3184120}).
\item\label{item:2} Construct the ideal of jets $\mathcal{J}_s (I(G))$
  of $I(G)$ of a certain order $s\in \mathbb{N}$.
\item\label{item:3} Take the radical $\sqrt{\mathcal{J}_s (I(G))}$ of
  the ideal $\mathcal{J}_s (I(G))$, which is a squarefree quadratic
  monomial ideal.
\item\label{item:4} Construct the simple undirected graph
  $\mathcal{J}_s (G)$ associated with the ideal
  $\sqrt{\mathcal{J}_s (I(G))}$.
\end{enumerate}
As observed by Goward and Smith, the ideal of jets constructed in step
\ref{item:2} is not necessarily a monomial ideal; however, its radical
$\sqrt{\mathcal{J}_s (I(G))}$, constructed in step \ref{item:3}, is a
monomial ideal. It is also easy to see that
$\sqrt{\mathcal{J}_s (I(G))}$ is quadratic because the edge ideal
$I(G)$ from step \ref{item:1} is quadratic. This ensures that the
construction in step \ref{item:4} can be carried through, resulting
in a graph that we call the \emph{graph of $s$-jets} of $G$.

There are many fruitful connections between combinatorics and algebra,
and more specifically between graph theory and commutative algebra
(see, for example, \cite{MR3184120}). We believe it would be
interesting to further develop these connections so that they can be
applied to the study of jets of monomial ideals. The present work
represents a preliminary step in this direction. We define jets of
graphs and explore some graph theoretic aspects of this construction,
with a focus on properties that are significant because of their
relation with commutative algebra. The article is structured as
follows. Section \ref{sec:backgr-defin} defines jet graphs and records
simple observations used throughout the rest of the paper. In section
\ref{sec:chromatic-number-jet}, we study the chromatic number of jet
graphs and prove the following (see Theorem \ref{thm:2}).
\begin{theorem_no_number}
  The graph $G$ has chromatic number $c$ if and only if for every
  $s\in \mathbb{N}$ $\mathcal{J}_s (G)$ has chromatic number $c$.
\end{theorem_no_number}
In section \ref{sec:when-are-jet}, we turn our attention to co-chordal
graphs, and we are able to prove the following general result (see
Theorem \ref{thm:3}).
\begin{theorem_no_number}
  If a graph $G$ has diameter greater than or equal to 3, then for
  every integer $s \geqslant 1$ the graph $\mathcal{J}_s(G)$ is not
  co-chordal.
\end{theorem_no_number}
As for graphs with diameter smaller than 3, we are able to show that
all jet graphs of complete graphs and star graphs are co-chordal (see
Propositions \ref{pro:1} and \ref{pro:2}). Finally, we turn our
attention to vertex covers of jet graphs in Section
\ref{sec:vertex-covers-jet}. There we provide some ways to construct
vertex covers of jet graphs from covers of the original graphs (see
Propositions \ref{pro:3} and \ref{pro:4}). We also show that all jets
of a complete bipartite graph $K_{n,n}$ are very well covered (i.e.,
all their vertex covers contain half of the vertices). Wherever
applicable throughout the paper, we highlight connections between
commutative algebra and graph theory. Section \ref{sec:open-questions}
contains some open problems and closing remarks.

The research leading to this manuscript was conducted in the summer of
2020. The second and third author were funded by an Undergraduate
Summer Research Award sponsored by the Office of Research at Cleveland
State University; all three authors are very grateful for this
opportunity. The first author also wishes to thank Adam Van Tuyl for
many useful suggestions in the planning phase of this project. We are
indebted to the anonymous referees who provided valuable feedback and
helped improve the quality of this work.

\section{Background and Definitions}
\label{sec:backgr-defin}

We start by recalling the definition of jets. Consider the polynomial
ring $R = \Bbbk [x_1,\dots,x_n]$ with coefficients in a field
$\Bbbk$. The choice of field $\Bbbk$ has no bearing on the content of
this work. Let $I = \langle f_1,\dots,f_r\rangle$ be an ideal in
$R$. For $s\in \mathbb{N}$, we define a new polynomial ring
\begin{equation*}
  \mathcal{J}_s (R) = \Bbbk [x_{i,j} \,|\, i=1,\dots,n, j=0,\dots,s]
\end{equation*}
with coefficients in the same field $\Bbbk$ and $(s+1)n$ variables
$x_{i,j}$. For each $i=1,\dots,n$, we will perform the substitution
\begin{equation*}
  x_i \mapsto x_{i,0} + x_{i,1} t + x_{i,2} t^2 + \dots + x_{i,s} t^s
  = \sum_{j=0}^s x_{i,j} t^j.
\end{equation*}
This substitution takes elements of $R$ to elements of
$\mathcal{J}_s (R) [t]$. Applying this substitution to a generator
$f_i$ of $I$ gives the following sum decomposition:
\begin{equation*}
  f_i \left( \sum_{j=0}^s x_{1,j} t^j, \dots, \sum_{j=0}^s x_{n,j} t^j \right)
  = \sum_{j\geqslant 0} f_{i,j} t^j,
\end{equation*}
where the coefficients $f_{i,j}$ are polynomials in
$\mathcal{J}_s (R)$. The \emph{ideal of $s$-jets} of
$I = \langle f_1,\dots,f_r\rangle$ is the ideal in the polynomial ring
$\mathcal{J}_s (R)$ defined by
\begin{equation*}
  \mathcal{J}_s (I) = \langle f_{i,j} \,|\,
  i=1,\dots,r, j=0,\dots,s\rangle.
\end{equation*}
For more details on the construction and motivation behind jets, the
reader may consult \cite{MR2483946}.

We are interested in studying jets of edge ideals of graphs. For a
simple graph $G$ on vertices $\{x_1,\dots,x_n\}$, the \emph{edge ideal}
$I(G)$ is the ideal in $R = \Bbbk [x_1,\dots, x_n]$ generated by the
squarefree monomials of degree two corresponding to edges of $G$,
i.e.:
\begin{equation*}
  I(G) = \langle x_i x_j \,|\, \{x_i,x_j\} \text{ is an edge of } G\rangle.
\end{equation*}
For an introduction to edge ideals, we recommend \cite{MR3184120}.

\begin{remark}
  To make the notation for jets less cumbersome, one may label
  vertices of a graph $G$ by different letters $a,b,c,\dots$ without
  subscripts. In this case, the edge ideal $I(G)$ is a subset of the
  polynomial ring $R = \Bbbk [a,b,c,\dots]$. Applying jets in this
  scenario produces an ideal $\mathcal{J}_s (I(G))$ in the polynomial
  ring
  $\mathcal{J}_s (R) = \Bbbk [a_j,b_j,c_j,\dots \,|\,
  j=0,\dots,s]$. The variable substitution explained above takes the
  form
  \begin{equation*}
    a \mapsto a_0 + a_1 t + a_2 t^2 + \dots + a_s t^s
    = \sum_{j=0}^s a_j t^j,
  \end{equation*}
  and similarly for other letters. We will use this notation in many
  of our examples.
\end{remark}

\begin{example}\label{exa:1}
  Consider the connected simple graphs on three vertices: the path of
  length two, which we denote by $P_3$, and the complete graph $K_3$ (see
  Figures \ref{fig:1} and \ref{fig:2} for a labeling of the
  vertices). The edge ideals for these two graphs are
  $I(P_3)=\langle ac,cb \rangle$ and
  $I(K_3)=\langle ac,cb,ab\rangle$. From here, we are able
  to find the first order jets for both ideals by replacing $a$ with
  $a_0+a_1t$, $b$ with $b_0+b_1t$, and $c$ with $c_0+c_1t$, then
  extracting coefficients of powers of $t$. This procedure gives the
  ideals
  \begin{equation*}
    \begin{split}
      \mathcal{J}_1 (I(P_3)) &= \langle a_0c_0,b_0c_0,c_0a_1+a_0c_1,c_0b_1+b_0c_1 \rangle,\\
      \mathcal{J}_1 (I(K_3)) &= \langle
      a_0b_0,a_0c_0,b_0c_0,b_0a_1+a_0b_1,c_0a_1+a_0c_1,c_0b_1+b_0c_1\rangle.
    \end{split}
  \end{equation*}
  However, these ideals are not squarefree monomial ideals, hence they
  do not correspond to graphs in an obvious way.
\end{example}

As the example shows, jets of squarefree monomial ideals need not be
monomial ideals. In particular, there is no obvious way to associate
back a graph to the jets of the edge ideal of a given graph. To work
around this issue, we appeal to a result of Goward and Smith
\cite[Theorem 2.1]{MR2229478}.

\begin{theorem}\label{thm:1}
  Let $I$ be an ideal generated by squarefree monomials in the
  variables $x_1,\dots, x_n$.  For every $s\in \mathbb{N}$, the
  radical $\sqrt{\mathcal{J}_s (I)}$ is a squarefree monomial ideal in
  the variables $x_{i,j}$ for $i=1,\dots,n$ and
  $j=0,\dots,s$. Moreover, $\sqrt{\mathcal{J}_s (I)}$ is minimally
  generated by the monomials
  \begin{equation*}
    x_{i_1,j_1} \cdots x_{i_r,j_r}
  \end{equation*}
  such that $x_{i_1} \cdots x_{i_r}$ is a minimal generator of $I$ and
  $\sum_{k=1}^r j_k \leqslant s$.
\end{theorem}

Note that if $I$ is generated by squarefree monomials of degree two,
then so is $\sqrt{\mathcal{J}_s (I)}$ for every $s\in \mathbb{N}$.

\begin{example}
  Taking radicals of the first order jets of the edge ideals in
  Example \ref{exa:1}, gives
  \begin{equation*}
    \begin{split}
      \sqrt{\mathcal{J}_1 (I(P_3))}&=\langle a_0c_0,b_0c_0,c_0a_1,a_0c_1,c_0b_1,b_0c_1 \rangle,\\
      \sqrt{\mathcal{J}_1 (I(K_3))}&=\langle a_0b_0,a_0c_0,b_0c_0,b_0a_1,a_0b_1,c_0a_1,a_0c_1,c_0b_1,b_0c_1\rangle.
    \end{split}
  \end{equation*}
  Figures \ref{fig:1} and \ref{fig:2} display the graphs corresponding
  to these squarefree quadratic monomial ideals.
\end{example}

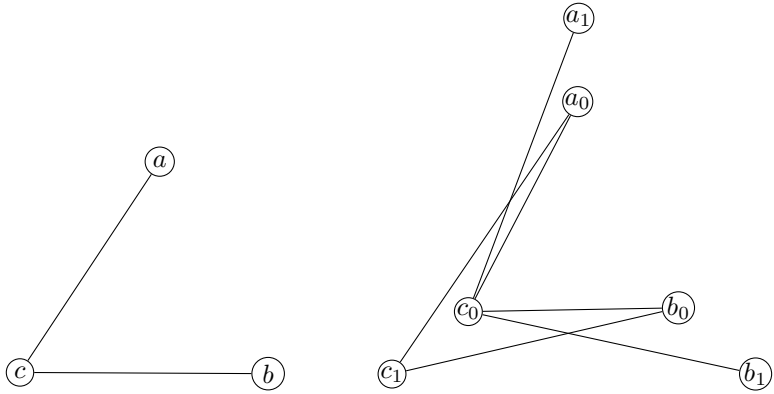
\begin{figure}[htb]
\begin{tikzpicture}     \newcommand*\pointsWeLEo{566.9366315835988/277.8409586235952/0/a,474.06586573071047/418.038945160875/1/c,638.9828751279186/418.8409038867429/2/b}          \newcommand*\edgesWeLEo{1/0,2/1}          \newcommand*\scaleWeLEo{0.02}          \foreach \x/\y/\z/\w in \pointsWeLEo {          \node (\z) at (\scaleWeLEo*\x,-\scaleWeLEo*\y) [circle,draw,inner sep=1.5pt] {$\w$};          }          \foreach \x/\y in \edgesWeLEo {          \draw (\x) -- (\y);          }      \end{tikzpicture} \hspace{1cm}
\begin{tikzpicture}       \newcommand*\pointsdoKKz{588.8912800591662/325.3219149507148/0/a_0,516.1884537622776/464.9783102959501/1/c_0,655.8251743033071/462.46722220023736/2/b_0,589.5964132562748/269.8801883292826/3/a_1,707.0686335438787/506.4631455640356/4/b_1,465.4554157015077/506.5240585769571/5/c_1}          \newcommand*\edgesdoKKz{1/0,2/1,3/1,4/1,5/0,5/2}          \newcommand*\scaledoKKz{0.02}          \foreach \x/\y/\z/\w in \pointsdoKKz {          \node (\z) at (\scaledoKKz*\x,-\scaledoKKz*\y) [circle,draw,inner sep=0pt] {$\w$};          }          \foreach \x/\y in \edgesdoKKz {          \draw (\x) -- (\y);          }      \end{tikzpicture}     
\caption{The graphs $P_3$ and $\mathcal{J}_1(P_3)$} \label{fig:1}
\end{figure}

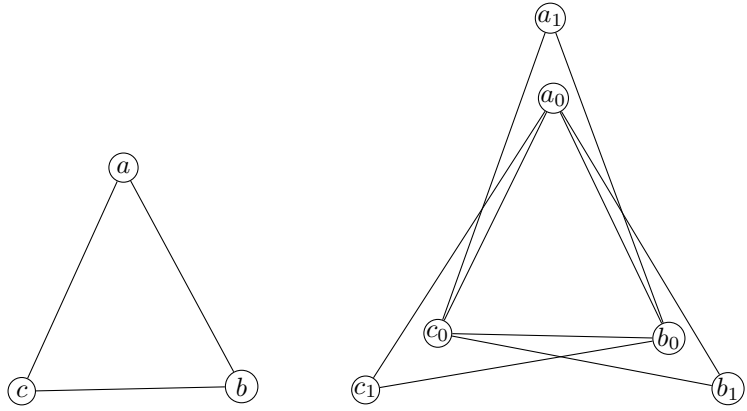
\begin{figure}[htb]
\begin{tikzpicture}     \newcommand*\pointstPRHe{613.9190294092147/306.2920128334028/0/a,692.4325245327401/451.87839756200975/1/b,546.2990322109628/454.9940547734596/2/c}          \newcommand*\edgestPRHe{1/0,2/0,2/1}          \newcommand*\scaletPRHe{0.02}          \foreach \x/\y/\z/\w in \pointstPRHe {          \node (\z) at (\scaletPRHe*\x,-\scaletPRHe*\y) [circle,draw,inner sep=1.5pt] {$\w$};          }          \foreach \x/\y in \edgestPRHe {          \draw (\x) -- (\y);          }      \end{tikzpicture} \hspace{1cm}
\begin{tikzpicture}        \newcommand*\pointsoEFxf{592.584530749817/299.49999729142974/0/a_0,669.3167619986655/459.1990050242759/1/b_0,515.7449874555851/455.87068533586995/2/c_0,590.532527552619/246.26547635844594/3/a_1,708.615469387584/492.66710822946135/4/b_1,467.77107823401235/493.532620250375/5/c_1}          \newcommand*\edgesoEFxf{1/0,2/0,2/1,3/1,3/2,4/0,4/2,5/0,5/1}          \newcommand*\scaleoEFxf{0.02}          \foreach \x/\y/\z/\w in \pointsoEFxf {          \node (\z) at (\scaleoEFxf*\x,-\scaleoEFxf*\y) [circle,draw,inner sep=0pt] {$\w$};          }          \foreach \x/\y in \edgesoEFxf {          \draw (\x) -- (\y);          }      \end{tikzpicture} \caption{The graphs $K_3$ and $\mathcal{J}_1(K_3)$}     \label{fig:2}
\end{figure}

In light of the discussion above, we make the following definition.

\begin{definition}
  Let $G$ be a simple graph. For $s\in \mathbb{N}$, the graph
  corresponding to the squarefree quadratic monomial ideal
  $\sqrt{\mathcal{J}_s (I(G))}$ is called the \emph{graph of $s$-jets}
  of $G$ or the \emph{$s$-jet graph} of $G$, and is denoted
  $\mathcal{J}_s (G)$.
\end{definition}

Note that $\mathcal{J}_0 (G)$ is isomorphic to $G$. The following
result gives a simple description of the edges in the $s$-jets of a
graph.

\begin{lemma}\label{lem:1}
  Let $G$ be a simple graph and let $a,b$ be distinct vertices of
  $G$. For every $s\in \mathbb{N}$, the set $\{a_i,b_j\}$ is an edge
  in $\mathcal{J}_s (G)$ if and only if $\{a,b\}$ is an edge of $G$
  and $i+j \leqslant s$.
\end{lemma}

\begin{proof}
  The statement is an immediate consequence of Theorem \ref{thm:1}.
\end{proof}

\begin{remark}
  Notice that if $a$ is a vertex in a simple graph $G$, then the graph
  $\mathcal{J}_s (G)$ contains no edge of the form $\{a_i,a_j\}$ for
  $i,j=0,\dots,s$.
\end{remark}

The next result establishes a basic property of jet graphs.

\begin{proposition}
  Let $G$ be a connected simple graph. For every $s\in \mathbb{N}$,
  the graph $\mathcal{J}_s (G)$ is connected.
\end{proposition}

\begin{proof}
  Let $a_i$ and $b_j$ be vertices in $\mathcal{J}_s (G)$ corresponding
  to vertices $a$ and $b$ in $G$. If $a$ and $b$ are adjacent in $G$,
  then the edges $\{a_i, b_0\}$, $\{b_0,a_0\}$, $\{a_0, b_j\}$ form a
  path joining $a_i$ and $b_j$.  If $a$ and $b$ are not adjacent in
  $G$, then there are vertices $v_1,\dots,v_r$ in $G$ and edges
  $\{a,v_1\}$, $\{v_1,v_2\}$, \dots, $\{v_{r-1},v_r\}$, $\{v_r, b\}$
  forming a path between $a$ and $b$ in $G$. Then there are edges
  $\{a_i,v_{1,0}\}$, $\{v_{1,0},v_{2,0}\}$, \dots,
  $\{v_{r-1,0},v_{r,0}\}$, $\{v_{r,0}, b_j\}$ connecting $a_i$ and
  $b_j$ in $\mathcal{J}_s (G)$.
\end{proof}

From now on, all graphs we consider are simple and connected.

\section{Chromatic Number of Jet Graphs}
\label{sec:chromatic-number-jet}

\begin{theorem}\label{thm:2}
  The graph $G$ has chromatic number $c$ if and only if for every
  $s\in \mathbb{N}$ $\mathcal{J}_s (G)$ has chromatic number $c$.
\end{theorem}

\begin{proof}
  Suppose for every $s\in \mathbb{N}$ $\mathcal{J}_s (G)$ has
  chromatic number $c$. This includes $\mathcal{J}_0 (G)$ which is
  isomorphic to $G$. Therefore $G$ has chromatic number $c$.

  Now assume $G$ has chromatic number $c$. This means we can assign a
  coloring of the vertices $\{x_1,\dots,x_n\}$ of $G$ using $c$
  different colors so that any two adjacent vertices have different
  colors. We assign a coloring using the same $c$ colors on the set of
  vertices of $\mathcal{J}_s (G)$:
  \begin{equation}
    \{x_{i,j} \,|\, i=1,\dots,n, j=0,\dots,s\}
  \end{equation}
  by declaring that for all $j=0,\dots,s$ the vertices $x_{i,j}$ in
  $\mathcal{J}_s (G)$ have the same color as the vertex $x_i$ in
  $G$. We show this is a vertex coloring of $\mathcal{J}_s
  (G)$. Consider two vertices $x_{i,j}$ and $x_{k,l}$ that are
  adjacent in $\mathcal{J}_s (G)$. We need to show they have different
  colors. By Lemma \ref{lem:1}, $x_{i,j}$ and $x_{k,l}$ are adjacent
  in $\mathcal{J}_s (G)$ if and only if $x_i$ and $x_k$ are adjacent
  in $G$ and $j+l\leqslant s$. Therefore $x_{i,j}$ and $x_{k,l}$
  inherit the two different colors coming from $x_i$ and $x_j$. This
  shows that $\mathcal{J}_s (G)$ admits a vertex coloring with $c$
  different colors.
  
  To prove that $\mathcal{J}_s (G)$ has chromatic number $c$ we need
  to show that $\mathcal{J}_s (G)$ does not allow a vertex coloring
  with fewer than $c$ colors. This follows directly from the fact that
  $\mathcal{J}_s (G)$ contains $\mathcal{J}_0 (G)$ as an induced
  subgraph and $\mathcal{J}_0 (G)$ is isomorphic to $G$ which has
  chromatic number $c$.
\end{proof}

\begin{example}
  Here is an example of a coloring on the 4-cycle (Figure \ref{fig:3})
  and its jets of order 1 and 2 (Figures \ref{fig:4} and
  \ref{fig:5}). Notice that for $s \geqslant 1$ the vertices
  $a_s,b_s,c_s,d_s$ are assigned the same color as the corresponding
  vertices $a,b,c,d$ in the original graph. As shown in Theorem
  \ref{thm:2}, this assignment gives a coloring for all orders of jets
  and illustrates why jet graphs have the same chromatic number as
  their 0-jet graph.
\end{example}

\begin{figure}[htb]
  \centering
  \definecolor{wrwrwr}{rgb}{0.3803921568627451,0.3803921568627451,0.3803921568627451}
  \begin{tikzpicture}[line cap=round,line join=round,>=triangle 45,x=1cm,y=1cm]
    \clip(-11.594055585365698,4.398532333419545) rectangle (-7.794678297429547,6.70529711538077);
    \draw [line width=2pt,color=wrwrwr] (-10.88761904761904,6.40597402597402)-- (-10.871877213695388,4.705855962219594);
    \draw [line width=2pt,color=wrwrwr] (-8.730987800078702,4.705855962219594)-- (-10.871877213695388,4.705855962219594);
    \draw [line width=2pt,color=wrwrwr] (-10.88761904761904,6.40597402597402)-- (-8.746729634002354,6.40597402597402);
    \draw [line width=2pt,color=wrwrwr] (-8.746729634002354,6.40597402597402)-- (-8.730987800078702,4.705855962219594);
    \draw [fill=black] (-10.88761904761904,6.40597402597402) circle (2.5pt);
    \draw (-11.2,6.5) node {$a$};
    \draw [fill=black] (-8.730987800078702,4.705855962219594) circle (2.5pt);
    \draw (-8.4,4.65) node {$b$};
    \draw [fill=white] (-10.871877213695388,4.705855962219594) circle (2.5pt);
    \draw (-11.2,4.65) node {$c$};
    \draw [fill=white] (-8.746729634002354,6.40597402597402) circle (2.5pt);
    \draw (-8.4,6.5) node {$d$};
  \end{tikzpicture}
  \caption{Coloring of a 4-cycle}\label{fig:3}
\end{figure}
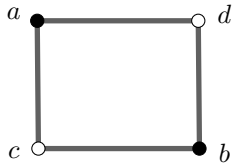

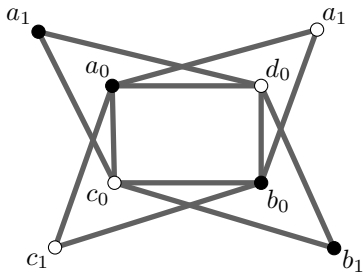
\begin{figure}[htb]
  \centering
  \definecolor{wrwrwr}{rgb}{0.3803921568627451,0.3803921568627451,0.3803921568627451}
  \hspace*{3em}\begin{tikzpicture}[line cap=round,line join=round,>=triangle 45,x=1cm,y=1cm, scale=.5]
    \clip(-11.969782761484788,-1.5) rectangle (-0.2603552522475032,5.617180029932563);
    \draw [line width=2pt,color=wrwrwr] (-8.92,3.56)-- (-8.86,0.98);
    \draw [line width=2pt,color=wrwrwr] (-4.96,0.98)-- (-8.86,0.98);
    \draw [line width=2pt,color=wrwrwr] (-4.96,3.56)-- (-8.92,3.56);
    \draw [line width=2pt,color=wrwrwr] (-4.96,3.56)-- (-4.96,0.98);
    \draw [line width=2pt,color=wrwrwr] (-8.86,0.98)-- (-10.88,5);
    \draw [line width=2pt,color=wrwrwr] (-4.96,3.56)-- (-10.88,5);
    \draw [line width=2pt,color=wrwrwr] (-8.86,0.98)-- (-3.02,-0.76);
    \draw [line width=2pt,color=wrwrwr] (-4.96,3.56)-- (-3.02,-0.76);
    \draw [line width=2pt,color=wrwrwr] (-8.92,3.56)-- (-10.44,-0.74);
    \draw [line width=2pt,color=wrwrwr] (-4.96,0.98)-- (-10.44,-0.74);
    \draw [line width=2pt,color=wrwrwr] (-8.92,3.56)-- (-3.4733079224650502,5.056188293862837);
    \draw [line width=2pt,color=wrwrwr] (-4.96,0.98)-- (-3.4733079224650502,5.056188293862837);
    \draw [fill=black] (-8.92,3.56) circle (5pt);
    \draw (-9.3,4) node {$a_0$};
    \draw [fill=black] (-4.96,0.98) circle (5pt);
    \draw (-4.5,0.5) node {$b_0$};
    \draw [fill=white] (-8.86,0.98) circle (5pt);
    \draw (-9.3,0.6) node {$c_0$};
    \draw [fill=white] (-4.96,3.56) circle (5pt);
    \draw (-4.5,4) node {$d_0$};
    \draw [fill=black] (-10.88,5) circle (5pt);
    \draw (-11.4,5.4) node {$a_1$};
    \draw [fill=white] (-10.44,-0.74) circle (5pt);
    \draw (-10.9,-1.1) node {$c_1$};
    \draw [fill=black] (-3.02,-0.76) circle (5pt);
    \draw (-2.5,-1.1) node {$b_1$};
    \draw [fill=white] (-3.4733079224650502,5.056188293862837) circle (5pt);
    \draw (-3,5.5) node {$d_1$};
  \end{tikzpicture}
  \caption{Coloring of first jets of a 4-cycle}\label{fig:4}
\end{figure}
\begin{figure}[htb]
  \centering
  \definecolor{wrwrwr}{rgb}{0.3803921568627451,0.3803921568627451,0.3803921568627451}
  \definecolor{rvwvcq}{rgb}{0.08235294117647059,0.396078431372549,0.7529411764705882}
  \definecolor{dtsfsf}{rgb}{0.8274509803921568,0.1843137254901961,0.1843137254901961}
  \hspace*{1em}\begin{tikzpicture}[line cap=round,line join=round,>=triangle 45,x=1cm,y=1cm, scale=0.5]
    \clip(-10,-5.5) rectangle (8,6.87);
    \draw [line width=2pt,color=wrwrwr] (-4.4,2.36)-- (-4.48,-0.94);
    \draw [line width=2pt,color=wrwrwr] (0.92,-0.96)-- (-4.48,-0.94);
    \draw [line width=2pt,color=wrwrwr] (-4.4,2.36)-- (0.88,2.36);
    \draw [line width=2pt,color=wrwrwr] (0.88,2.36)-- (0.92,-0.96);
    \draw [line width=2pt,color=wrwrwr] (-4.48,-0.94)-- (-6.3,3.78);
    \draw [line width=2pt,color=wrwrwr] (0.88,2.36)-- (-6.3,3.78);
    \draw [line width=2pt,color=wrwrwr] (-4.48,-0.94)-- (2.72,-2.88);
    \draw [line width=2pt,color=wrwrwr] (0.88,2.36)-- (2.72,-2.88);
    \draw [line width=2pt,color=wrwrwr] (-4.4,2.36)-- (-6.3,-2.88);
    \draw [line width=2pt,color=wrwrwr] (0.92,-0.96)-- (-6.3,-2.88);
    \draw [line width=2pt,color=wrwrwr] (-4.4,2.36)-- (2.54,3.76);
    \draw [line width=2pt,color=wrwrwr] (0.92,-0.96)-- (2.54,3.76);
    \draw [line width=2pt,color=wrwrwr] (-6.3,3.78)-- (-6.3,-2.88);
    \draw [line width=2pt,color=wrwrwr] (-6.3,-2.88)-- (2.72,-2.88);
    \draw [line width=2pt,color=wrwrwr] (-6.3,3.78)-- (2.54,3.76);
    \draw [line width=2pt,color=wrwrwr] (2.54,3.76)-- (2.72,-2.88);
    \draw [line width=2pt,color=wrwrwr] (-4.48,-0.94)-- (-8.14,5.32);
    \draw [line width=2pt,color=wrwrwr] (0.88,2.36)-- (-8.14,5.32);
    \draw [line width=2pt,color=wrwrwr] (-4.48,-0.94)-- (5.04,-4.52);
    \draw [line width=2pt,color=wrwrwr] (0.88,2.36)-- (5.04,-4.52);
    \draw [line width=2pt,color=wrwrwr] (-4.4,2.36)-- (-7.6,-4.26);
    \draw [line width=2pt,color=wrwrwr] (0.92,-0.96)-- (-7.6,-4.26);
    \draw [line width=2pt,color=wrwrwr] (-4.4,2.36)-- (4.5,5.3);
    \draw [line width=2pt,color=wrwrwr] (0.92,-0.96)-- (4.5,5.3);
    \draw [fill=black] (-4.4,2.36) circle (5pt);
    \draw (-4.76,2.96) node {$a_0$};
    \draw [fill=black] (0.92,-0.96) circle (5pt);
    \draw (1.5,-1.3) node {$b_0$};
    \draw [fill=white] (-4.48,-0.94) circle (5pt);
    \draw (-4.84,-1.25) node {$c_0$};
    \draw [fill=white] (0.88,2.36) circle (5pt);
    \draw (1.1,2.84) node {$d_0$};
    \draw [fill=black] (-6.3,3.78) circle (5pt);
    \draw (-6.66,4.36) node {$a_1$};
    \draw [fill=black] (-8.14,5.32) circle (5pt);
    \draw (-8.62,5.92) node {$a_2$};
    \draw [fill=white] (-6.3,-2.88) circle (5pt);
    \draw (-6.6,-3.3) node {$c_1$};
    \draw [fill=white] (-7.6,-4.26) circle (5pt);
    \draw (-8.2,-4.6) node {$c_2$};
    \draw [fill=black] (2.72,-2.88) circle (5pt);
    \draw (3.3,-3.1) node {$b_1$};
    \draw [fill=black] (5.04,-4.52) circle (5pt);
    \draw (5.5,-5) node {$b_2$};
    \draw [fill=white] (2.54,3.76) circle (5pt);
    \draw (2.9,4.24) node {$d_1$};
    \draw [fill=white] (4.5,5.3) circle (5pt);
    \draw (4.72,5.78) node {$d_2$};
  \end{tikzpicture}
  \caption{Coloring of second jets of a 4-cycle}\label{fig:5}
\end{figure}
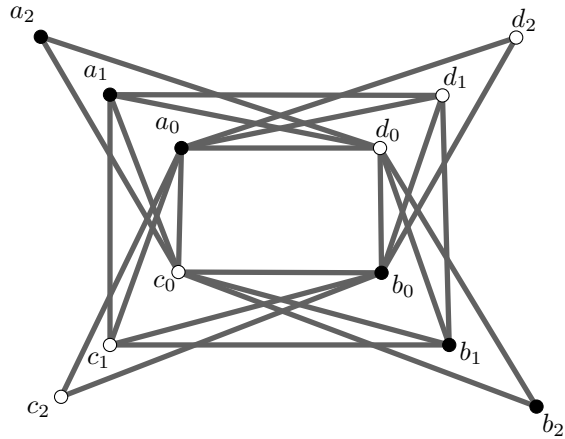

\begin{remark}
  The argument used to prove Theorem \ref{thm:2} can also be used to
  show that for an arbitrary positive integer $b$ and for every
  $s\in\mathbb{N}$ the graph $\mathcal{J}_s (G)$ has the same $b$-fold
  chromatic number as the graph $G$. It follows that
  $\mathcal{J}_s (G)$ also has the same fractional chromatic number as
  $G$. The $b$-fold chromatic number of a graph has an algebraic
  interpretation involving powers of the cover ideal of a graph (see
  \cite[Theorem 1.26]{MR3184120}), while the fractional chromatic
  number of a graph is related to the Waldschmidt constant of the edge
  ideal of the graph (see \cite[Theorem 1.1]{MR3566223}).
\end{remark}

\section{When are Jet Graphs Co-chordal?}
\label{sec:when-are-jet}

Next we look at the chordality of the complement of the jets of a
graph.  Recall that a graph is called \emph{chordal} (or
\emph{triangulated}) if it contains no induced cycle of length four or
more \cite[\S~9.7]{MR2368647}. We say a graph $G$ is \emph{co-chordal}
if its complement $G^c$ is chordal.  Co-chordal graphs are interesting
because their edge ideals are well behaved as expressed by the
following theorem.

\begin{theorem}[Fr\"oberg, \cite{MR1171260}]
  Let $G$ be a graph and $I(G)$ its edge ideal. Then $I(G)$ has a
  linear resolution if and only if $G$ is co-chordal.
\end{theorem}

Note that $I(G)$ having a linear resolution is equivalent to $I(G)$
having regularity 2. For details on resolutions of monomial ideals, we
invite the reader to consult \cite{MR2110098}.

It is natural to ask whether given a graph $G$ that is co-chordal, the
jets of $G$ will also be co-chordal. We found this does not hold true
in general.

\begin{example}\label{exa:2}
  Let $P_4$ be the path of length 3, which is co-chordal. The edge
  ideal of the path is $I (P_4) = \langle ab,bc,cd \rangle$, while its
  first and second jets have the following edge ideals:
  \begin{equation*}
    \begin{split}
      I(\mathcal{J}_1 (P_4)) = \langle
      &{a}_{0}{b}_{0},{b}_{0}{c}_{0},{c}_{0}{d}_{0},{a}_{0}{b}_{1},{a}_{1}{b}_{0},
      {b}_{0}{c}_{1},{b}_{1}{c}_{0},{c}_{0}{d}_{1},{c}_{1}{d}_{0},\rangle,\\
      I(\mathcal{J}_2 (P_4)) = \langle
      &{a}_{0}{b}_{0},{b}_{0}{c}_{0},{c}_{0}{d}_{0},{a}_{0}{b}_{1},{a}_{1}{b}_{0},
      {b}_{0}{c}_{1},{b}_{1}{c}_{0},{c}_{0}{d}_{1},{c}_{1}{d}_{0},\\
      &{a}_{0}{b}_{2},{a}_{1}{b}_{1},{a}_{2}{b}_{0},
      {b}_{0}{c}_{2},{b}_{1}{c}_{1},{b}_{2}{c}_{0},
      {c}_{0}{d}_{2},{c}_{1}{d}_{1},{c}_{2}{d}_{0}\rangle.
    \end{split}
  \end{equation*}

  Figure \ref{fig:6} shows the complement of the graph
  $\mathcal{J}_1 (P_4)$. We can see from the picture this graph is not
  chordal because it contains the induced cycle of length 4
  $(a_0 c_1 b_1 d_0)$. Therefore $\mathcal{J}_1 (P_4)$ is not
  co-chordal even though $P_4$ is.
  
  Similarly, $\mathcal{J}_2 (P_4)$ is not co-chordal. We do not offer
  a picture but we observe that the complement of
  $\mathcal{J}_2 (P_4)$ contains the induced cycle of length 4
  $(a_0 c_2 b_2 d_0)$. It is easy to see from the description of
  $I(\mathcal{J}_2 (P_4))$ that none of the edges of this cycle belong
  to $\mathcal{J}_2 (P_4)$ while its two chords $\{a_0, b_2\}$ and
  $\{c_2, d_0\}$ do.
\end{example}

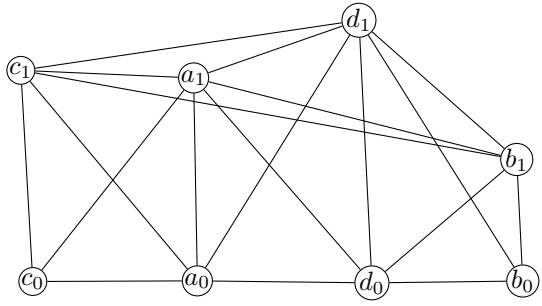
\begin{figure}[htb]
  \begin{tikzpicture}         \newcommand*\pointseRkKg{320.32992700657917/439.590748974379/0/c_0,429.6980514678469/439.17197741835355/1/a_0,312.3703167711542/299.1645165014176/2/c_1,427.20121386048976/304.20431176315776/3/a_1,537.0843396666834/265.8447924265863/4/d_1,545.7008227946919/440.5102298651267/5/d_0,641.9367337678449/358.45178039046385/6/b_1,645.9612031091742/439.27049499324835/7/b_0}          \newcommand*\edgeseRkKg{1/0,2/0,2/1,3/0,3/1,3/2,4/1,4/2,4/3,5/1,5/3,5/4,6/2,6/3,6/4,6/5,7/4,7/5,7/6}          \newcommand*\scaleeRkKg{0.02}          \foreach \x/\y/\z/\w in \pointseRkKg {          \node (\z) at (\scaleeRkKg*\x,-\scaleeRkKg*\y) [circle,draw,inner sep=0pt] {$\w$};          }          \foreach \x/\y in \edgeseRkKg {          \draw (\x) -- (\y);          }      \end{tikzpicture}     
    \caption{The Complement of $\mathcal{J}_1(P_4)$}\label{fig:6}
\end{figure}

Example \ref{exa:2} shows that jets of graphs do not necessarily
preserve the property of being co-chordal. In addition, the example
suggests it might be worth studying the co-chordality property of jets
in relation to the diameter of the original graph. Recall that the
\emph{diameter} of a connected graph is the maximum of the distances
between its vertices, where the \emph{distance} between two vertices
is defined as the number of edges in a shortest path connecting those
vertices \cite[\S~3.1]{MR2368647}. This leads to our first result on
co-chordality of jets of a graph.

\begin{theorem}\label{thm:3}
  If a graph $G$ has diameter greater than or equal to 3, then for
  every integer $s \geqslant 1$ the graph $\mathcal{J}_s(G)$ is not
  co-chordal.
\end{theorem}
\begin{proof}
  Let $G$ be a connected graph with diameter at least 3. By
  definition, there are two vertices of $G$ such that the shortest
  path between them has at least three edges. This implies that there
  are two vertices $a$ and $d$ along this path such that a shortest
  path between $a$ and $d$ has exactly three edges, say,
  $\{a,b\},\{b,c\},\{c,d\}$. Let $s\geqslant 1$ be an integer. To
  prove $\mathcal{J}_s (G)$ is not co-chordal we show
  $(a_0 d_0 b_s c_s)$ is a minimal cycle in $\mathcal{J}_s (G)^c$.

  First, we show the cycle $(a_0 d_0 b_s c_s)$ is contained in
  $\mathcal{J}_s (G)^c$. Notice that $\{a,d\}$, $\{b,d\}$, and
  $\{a,c\}$ are not edges of $G$ because if they were, then the
  distance between $a$ and $d$ would be less than 3 contradicting our
  assumption. Therefore $\{a_0,d_0\}$, $\{b_s ,d_0\}$, and
  $\{a_0,c_s\}$ are not edges of $\mathcal{J}_s (G)$ by Lemma
  \ref{lem:1}. Since $\{b,c\}$ is an edge in $G$, we cannot use the
  same reasoning to show $\{b_s,c_s\}$ is an edge in
  $\mathcal{J}_s (G)^c$. However, we can observe that the subscripts
  in $\{b_s,c_s\}$ add up to $2s>s$ (because $s\geqslant
  1$). Therefore Lemma \ref{lem:1} implies $\{b_s,c_s\}$ is not an
  edge in $\mathcal{J}_s (G)$.

  Next, we show the cycle $(a_0 d_0 b_s c_s)$ is an induced cycle of
  length four in $\mathcal{J}_s (G)^c$. It is enough to prove that the edges
  $\{a_0,b_s\}$ and $\{c_s, d_0\}$ (the only candidates for chords)
  are not in $\mathcal{J}_s (G)^c$. This follows once again from Lemma
  \ref{lem:1}: $\{a,b\}$ and $\{c,d\}$ are edges in $G$, and the
  subscripts in $\{a_0,b_s\}$ and $\{c_s, d_0\}$ add up to $s$, hence
  $\{a_0,b_s\}$ and $\{c_s, d_0\}$ are edges in $\mathcal{J}_s (G)$.

  Since $(a_0 d_0 b_s c_s)$ is an induced cycle of length 4 in
  $\mathcal{J}_s (G)^c$, we conclude that $\mathcal{J}_s (G)$ is not co-chordal.
\end{proof}

By contrast, the following result shows that graphs with diameter one
have co-chordal jets. Note that a connected graph with diameter one is
complete.

\begin{proposition}\label{pro:1}
  For all integers $n\geqslant 2$ and $s\geqslant 0$, the graphs
  $\mathcal{J}_s (K_n)$ are co-chordal.
\end{proposition}

We provide some details about our proof strategy. Recall that a
\emph{simplicial vertex} of a graph is a vertex whose neighbors induce
a complete subgraph. A \emph{simplicial order} of a graph is an
enumeration $x_1,x_2,\dots,x_n$ of its vertices such that for every
$i\in\{1,\dots,n\}$ the vertex $x_i$ is a simplicial vertex of the
graph induced on $x_i,\dots,x_n$.

\begin{example}
  Consider the graph in Figure \ref{fig:7}. The vertex $x_1$ is a
  simplicial vertex because its neighbors, $x_2$ and $x_4$, induce a
  complete subgraph. In fact, the enumeration of the vertices $x_1$,
  $x_2$, $x_3$, $x_4$ defines a simplicial order of the graph. By
  contrast, $x_2$ is not a simplicial vertex of the graph.
  
  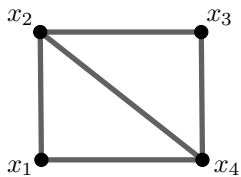
\begin{figure}[htb]
    \centering
    \definecolor{wrwrwr}{rgb}{0.3803921568627451,0.3803921568627451,0.3803921568627451}
        \begin{tikzpicture}[line cap=round,line join=round,>=triangle 45,x=1cm,y=1cm]
      \clip(-11.594055585365698,4.398532333419545) rectangle (-7.794678297429547,6.70529711538077);
      \draw [line width=2pt,color=wrwrwr] (-10.88761904761904,6.40597402597402)-- (-10.871877213695388,4.705855962219594);
      \draw [line width=2pt,color=wrwrwr] (-8.730987800078702,4.705855962219594)-- (-10.871877213695388,4.705855962219594);
      \draw [line width=2pt,color=wrwrwr] (-10.88761904761904,6.40597402597402)-- (-8.746729634002354,6.40597402597402);
      \draw [line width=2pt,color=wrwrwr] (-8.746729634002354,6.40597402597402)-- (-8.730987800078702,4.705855962219594);
      \draw [line width=2pt,color=wrwrwr] (-8.730987800078702,4.705855962219594)-- (-10.88761904761904,6.40597402597402);
      \draw[fill=black] (-10.88761904761904,6.40597402597402) circle (2.5pt);
      \draw (-11.15,6.6) node {$x_2$};
      \draw[fill=black] (-8.730987800078702,4.705855962219594) circle (2.5pt);
      \draw (-8.4,4.6) node {$x_4$};
      \draw[fill=black] (-10.871877213695388,4.705855962219594) circle (2.5pt);
      \draw (-11.15,4.6) node {$x_1$};
      \draw[fill=black] (-8.746729634002354,6.40597402597402) circle (2.5pt);
      \draw (-8.5,6.6) node {$x_3$};
    \end{tikzpicture}
    \caption{A simplicial order on a graph}\label{fig:7}
  \end{figure}

\end{example}

For more details on simplicial vertices and simplicial orders, we
invite the reader to consult \cite[\S~9.7]{MR2368647}. The following
result was first recorded in \cite[\S~7]{MR186421}, where chordal
graphs are referred to as ``rigid circuit graphs''; we are grateful to
Dr.~Robert Short for providing this reference.  The same result can
also be found in \cite[Corollary 9.22]{MR2368647}.

\begin{theorem}\label{thm:5}
  A graph is chordal if and only if it has a simplicial order.
\end{theorem}

We will use this characterization of chordal graphs throughout the
rest of this section.

\begin{proof}[Proof of Proposition \ref{pro:1}]
  Let $\{x_1,\dots,x_n\}$ be the set of vertices of $K_n$ and fix
  $s\geqslant 0$. The graph $\mathcal{J}_s (K_n)^c$ has set of vertices
  \begin{equation*}
    \{x_{i,t} \,|\, i=1,\dots,n, \, t=0,\dots,s\}.
  \end{equation*}
  The edges of $\mathcal{J}_s (K_n)^c$ fall into one of the following types:
  \begin{enumerate}[label=(\roman*)]
  \item\label{item:5} $\{x_{i,t}, x_{j,u}\}$ with $i\neq j$ and
    $t+u>s$ (by Lemma \ref{lem:1});
  \item\label{item:6} or $\{x_{i,t}, x_{i,u}\}$ with $t\neq u$.
  \end{enumerate}

  We order the vertices of $\mathcal{J}_s (K_n)^c$ as follows: we declare
  $x_{i,t} < x_{j,u}$ if and only if $t<u$, or $t=u$ and $i<j$. Note
  that this is a total order. We will show that ordering vertices this
  way defines a simplicial order on $\mathcal{J}_s (K_n)^c$.

  Fix a vertex $x_{i,t}$. We need to show $x_{i,t}$ is a simplicial
  vertex for the graph $G$ induced by $\mathcal{J}_s (K_n)^c$ on the vertices
  greater than or equal to $x_{i,t}$ in the order above. Let $x_{j,u}$
  and $x_{k,v}$ be distinct neighbors of $x_{i,t}$ in $G$; our goal is
  to show they are adjacent.
  \begin{description}[wide]
  \item[Case 1] If $j=k$, then $u\neq v$ and $\{x_{j,u}, x_{k,v}\}$ is
    an edge of type \ref{item:6}.
  \item[Case 2] If $j\neq k$ and $i\neq k$, then $t+v>s$ because
    $x_{k,v}$ is a neighbor of $x_{i,t}$.  Now $x_{i,t} < x_{j,u}$
    implies $t\leqslant u$. Therefore $u+v>s$ and
    $\{x_{j,u}, x_{k,v}\}$ is an edge of type \ref{item:5}.
  \item[Case 3] If $j\neq k$ and $i\neq j$, then the proof follows the
    same reasoning as the previous case.
  \end{description}

  This proves that $\mathcal{J}_s (K_n)^c$ has a simplicial order,
  therefore it is chordal by Theorem \ref{thm:5}.
\end{proof}

It remains to discuss the situation for graphs with diameter equal to
2. As we will show, the results depend on the graph.
\begin{example}\label{exa:3}
  The graph in Figure \ref{fig:8} has diameter 2 and is co-chordal, as
  seen through its complement.
  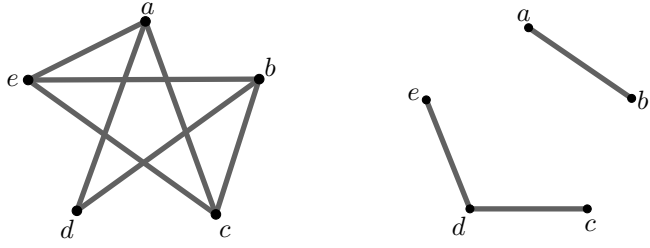
\begin{figure}[htb]
    \centering
    \definecolor{wrwrwr}{rgb}{0.3803921568627451,0.3803921568627451,0.3803921568627451}
    \definecolor{rvwvcq}{rgb}{0.08235294117647059,0.396078431372549,0.7529411764705882}
    \begin{tikzpicture}[line cap=round,line join=round,>=triangle 45,x=1cm,y=1cm,scale=.6]
      \clip(-6.447074355553356,0.3) rectangle (1.549955425144903,5.8);
      \draw [line width=2pt,color=wrwrwr] (-2.76,5.3)-- (-1.2,1.02);
      \draw [line width=2pt,color=wrwrwr] (-1.2,1.02)-- (-0.24,4.02);
      \draw [line width=2pt,color=wrwrwr] (-2.76,5.3)-- (-4.28,1.1);
      \draw [line width=2pt,color=wrwrwr] (-4.28,1.1)-- (-0.24,4.02);
      \draw [line width=2pt,color=wrwrwr] (-2.76,5.3)-- (-5.36,4);
      \draw [line width=2pt,color=wrwrwr] (-5.36,4)-- (-0.24,4.02);
      \draw [line width=2pt,color=wrwrwr] (-1.2,1.02)-- (-5.36,4);
      \draw [fill=black] (-2.76,5.3) circle (3pt);
      \draw (-2.7063034563939787,5.6) node {$a$};
      \draw [fill=black] (-0.24,4.02) circle (3pt);
      \draw (0,4.3) node {$b$};
      \draw [fill=black] (-1.2,1.02) circle (3pt);
      \draw (-1,0.65) node {$c$};
      \draw [fill=black] (-4.28,1.1) circle (3pt);
      \draw (-4.5,0.7) node {$d$};
      \draw [fill=black] (-5.36,4) circle (3pt);
      \draw (-5.7,4) node {$e$};
    \end{tikzpicture}
    \begin{tikzpicture}[line cap=round,line join=round,>=triangle 45,x=1cm,y=1cm,scale=.5]
      \clip(-6.221204015276286,-0.5694740407533441) rectangle (4.423418287875726,5.893332357588951);
      \draw [line width=2pt,color=wrwrwr] (-3,0.44)-- (-4.16,3.34);
      \draw [line width=2pt,color=wrwrwr] (-1.44,5.26)-- (1.3,3.38);
      \draw [line width=2pt,color=wrwrwr] (0.12,0.44)-- (-3,0.44);
      \draw [fill=black] (-1.44,5.26) circle (3pt);
      \draw (-1.5664998374171315,5.6) node {$a$};
      \draw [fill=black] (1.3,3.38) circle (3pt);
      \draw (1.6,3.3388084512140556) node {$b$};
      \draw [fill=black] (0.12,0.44) circle (3pt);
      \draw (0.2,0) node {$c$};
      \draw [fill=black] (-3,0.44) circle (3pt);
      \draw (-3.3,0) node {$d$};
      \draw [fill=black] (-4.16,3.34) circle (3pt);
      \draw (-4.5,3.5) node {$e$};
    \end{tikzpicture}
    \caption{A graph with diameter 2 and its complement}\label{fig:8}
  \end{figure}

  The first jets of the graph are also co-chordal. To see this, first
  compute the generators of the edge ideal.
  \begin{equation*}
    \begin{split}
      I(\mathcal{J}_1 (G)) = \langle
      &c_0e_1, b_0e_1, a_0e_1, b_0d_1, a_0d_1, e_0c_1, b_0c_1, a_0c_1,\\
      &e_0b_1, d_0b_1, c_0b_1, e_0a_1, d_0a_1, c_0a_1, c_0e_0, b_0e_0,\\
      &a_0e_0, b_0d_0, a_0d_0, b_0c_0, a_0c_0 \rangle
    \end{split}
  \end{equation*}
  Then observe that $a_0,b_0,c_0,e_0,d_0,a_1,b_1,c_1,d_1,e_1$ is a
  simplicial order of $\mathcal{J}_1 (G)$.
  
  By contrast, the second jets of the graph are not co-chordal.  It is
  possible to see from the generators of the edge ideal of the second
  jets that $(d_0 c_1 a_2 e_1)$ is an induced cycle of length four in
  the complement.
  \begin{equation*}
    \begin{split}
      I(\mathcal{J}_2(G)) = \langle
      &c_0e_2, b_0e_2, a_0e_2, b_0d_2, a_0d_2, e_0c_2, b_0c_2, a_0c_2, e_0b_2,\\
      &d_0b_2, c_0b_2, e_0a_2, d_0a_2, c_0a_2, c_1e_1, b_1e_1, a_1,e_1, c_0e_1,\\
      &b_0e_1, a_0e_1, b_1d_1, a_1d_1, b_0d_1, a_0d_1, b_1c_1, a_1c_1, e_0c_1,\\
      &b_0c_1, a_0c_1, e_0b_1, d_0b_1, c_0b_1, e_0a_1, d_0a_1, c_0a_1, c_0e_0,\\
      &b_0e_0, a_0e_0, b_0d_0, a_0d_0, b_0c_0, a_0c_0\rangle
    \end{split}
  \end{equation*}
\end{example}

For our next result, we consider the complete bipartite graphs
$K_{1,n}$, also known as stars \cite[\S 1.1, page 4]{MR2368647}. We
label the vertices $x_0,x_1,\dots,x_n$, with the bipartition given by
the subsets $\{x_0\}$ and $\{x_1,\dots,x_n\}$. In other words, $x_0$
is the ``internal'' vertex of the star (see Figure \ref{fig:9}).

\begin{figure}[htb]
  \definecolor{wrwrwr}{rgb}{0.3803921568627451,0.3803921568627451,0.3803921568627451}
  \hspace*{5em}\begin{tikzpicture}[line cap=round,line join=round,>=triangle 45,x=1cm,y=1cm,scale=.4]
    \clip(-8.225214124718262,-4.2330202854996175) rectangle (9.024973703981958,6.5);
    \draw [line width=2pt,color=wrwrwr] (-1.72,1.05)-- (-1.8509090909090922,5.7318181818181815);
    \draw [line width=2pt,color=wrwrwr] (-1.72,1.05)-- (2.92,3.21);
    \draw [line width=2pt,color=wrwrwr] (-1.72,1.05)-- (1.84,-3.25);
    \draw [line width=2pt,color=wrwrwr] (-1.72,1.05)-- (-5.38,-3.29);
    \draw [line width=2pt,color=wrwrwr] (-1.72,1.05)-- (-6.832727272727274,2.8045454545454542);
    \draw [fill=black] (-1.72,1.05) circle (4pt);
    \draw (-1.7188279489106002,0) node {$x_0$};
    \draw [fill=black] (-1.8509090909090922,5.7318181818181815) circle (4pt);
    \draw (-1.6887753568745372,6.15) node {$x_1$};
    \draw [fill=black] (2.92,3.21) circle (4pt);
    \draw (3.314981217129969,3.65) node {$x_2$};
    \draw [fill=black] (1.84,-3.25) circle (4pt);
    \draw (2.789060856498865,-3.158640120210362) node {$x_3$};
    \draw [fill=black] (-5.38,-3.29) circle (4pt);
    \draw (-6.286821938392191,-3.2938767843726455) node {$x_4$};
    \draw [fill=black] (-6.832727272727274,2.8045454545454542) circle (4pt);
    \draw (-7.473899323816684,3.1824567993989508) node {$x_5$};
  \end{tikzpicture}
  \caption{Star Graph $K_{1,5}$}
  \label{fig:9}
\end{figure}
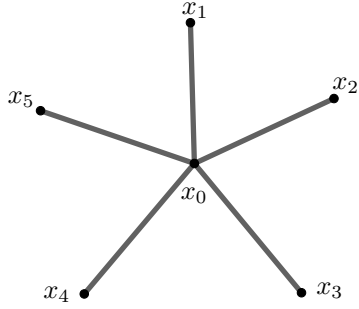

\begin{proposition}\label{pro:2}
  For all integers $n\geqslant 2$ and $s\geqslant 0$, the graphs
  $\mathcal{J}_s (K_{1,n})$ are co-chordal.
\end{proposition}

\begin{proof}
  Let $\{x_0, x_1,\dots,x_n\}$ be the set of vertices of $K_{1,n}$,
  partitioned as $\{x_0\}$ and $\{x_1,\dots,x_n\}$ (i.e., $x_0$ is the
  ``internal'' vertex of the star).  Fix $s\geqslant 0$. The graph
  $\mathcal{J}_s (K_{1,n})^c$ has set of vertices
  \begin{equation*}
    \{x_{i,t} \,|\, i=0,\dots,n, \, t=0,\dots,s\}.
  \end{equation*}
  The edges of $\mathcal{J}_s (K_{1,n})^c$ fall into one of the
  following types:
  \begin{enumerate}[label=(\roman*)]
  \item\label{item:7} $\{x_{0,t}, x_{j,u}\}$ with $j\neq 0$ and
    $t+u>s$ (by Lemma \ref{lem:1});
  \item\label{item:8} $\{x_{i,t}, x_{i,u}\}$ with $t\neq u$;
  \item\label{item:9} or $\{x_{i,t}, x_{j,u}\}$ with $i\neq 0$,
    $j\neq 0$, and $i\neq j$.
  \end{enumerate}

  We order the vertices of $\mathcal{J}_s (K_{1,n})^c$ as follows: we
  declare $x_{i,t} < x_{j,u}$ if and only if $t<u$, or $t=u$ and
  $i<j$. Note that this is a total order. We will show that ordering
  vertices this way defines a simplicial order on
  $\mathcal{J}_s (K_{1,n})^c$.

  Fix a vertex $x_{i,t}$. We need to show $x_{i,t}$ is a simplicial
  vertex for the graph $G$ induced by $\mathcal{J}_s (K_{1,n})^c$ on
  the vertices greater than or equal to $x_{i,t}$ in the order
  above. Let $x_{j,u}$ and $x_{k,v}$ be distinct neighbors of
  $x_{i,t}$ in $G$; our goal is to show they are adjacent.
  \begin{description}[wide]
  \item[Case 1] If $j=k$, then $u\neq v$ and $\{x_{j,u}, x_{k,v}\}$ is
    an edge of type \ref{item:8}.
  \item[Case 2] If $j\neq k$, with $j\neq 0$ and $k\neq 0$, then
    $\{x_{j,u}, x_{k,v}\}$ is an edge of type \ref{item:9}.
  \item[Case 3] If $j\neq k$, with $j=0$, then $k\neq 0$ and we can
    consider two different scenarios.
    \begin{itemize}
    \item If $i=0$, then $t+v>s$ because $x_{k,v}$ is a neighbor of
      $x_{i,t}$. However, $x_{i,t} < x_{j,u}$ implies $t\leqslant u$.
    \item If $i\neq 0$, then $t+u>s$ because $x_{j,u}$ is a neighbor
      of $x_{i,t}$. However, $x_{i,t} < x_{k,v}$ implies
      $t\leqslant v$.
    \end{itemize}
    Either way, $u+v>s$ and $\{x_{j,u}, x_{k,v}\}$ is an edge of type
    \ref{item:9}.
  \item[Case 4] If $j\neq k$, with $k=0$, then $j\neq 0$ and the proof
    follows the same reasoning as the previous case.
  \end{description}

  This proves that $\mathcal{J}_s (K_{1,n})^c$ has a simplicial order,
  therefore it is chordal by Theorem \ref{thm:5}.
\end{proof}

\section{Vertex Covers of Jet Graphs}
\label{sec:vertex-covers-jet}

Recall that a subset $C$ of the set of vertices of a graph $G$ is
called a \emph{vertex cover} if every edge of $G$ has an endpoint in
$C$. In addition, $C$ is a \emph{minimal} vertex cover if no proper
subset of $C$ is a vertex cover. From the point of view of
combinatorial commutative algebra, minimal vertex covers of a graph
correspond to minimal generators of the cover ideal of the graph and
to components in the irreducible decomposition of the edge ideal of
the graph. For more details, see \cite[\S 1]{MR3184120}.

We are interested in describing minimal vertex covers of the jets of a
graph and how they relate to vertex covers of the original graph.

\begin{example}\label{exa:4}
  Consider the complete graph $K_3$ with vertex set $\{a,b,c\}$ (see
  Figure \ref{fig:2} for a picture of $K_3$ and its first jets). The
  minimal vertex covers of $K_3$ are $\{a,b\}$, $\{a,c\}$, and
  $\{b,c\}$. Using the Macaulay2 \cite{M2} package \texttt{EdgeIdeals}
  \cite{MR2878668,EdgeIdealsSource}, we found the following
  information. The minimal vertex covers of $\mathcal{J}_1 (K_3)$ are:
  \begin{equation*}
    \{a_0,b_0,c_0\}, \{a_0,b_0,a_1,b_1\}, \{a_0,c_0,a_1,c_1\}, \{b_0,c_0,b_1,c_1\}.
  \end{equation*}
  The minimal vertex covers of $\mathcal{J}_2 (K_3)$ are:
  \begin{equation*}
    \begin{split}
      &\{a_0,b_0,c_0,a_1,b_1\}, \{a_0,b_0,c_0,a_1,c_1\}, \{a_0,b_0,c_0,b_1,c_1\},\\
      &\{a_0,b_0,a_1,b_1,a_2,b_2\}, \{a_0,c_0,a_1,c_1,a_2,c_2\}, \{b_0,c_0,b_1,c_1,b_2,c_2\}.
    \end{split}
  \end{equation*}
  The minimal vertex covers of $\mathcal{J}_3 (K_3)$ are:
  \begin{equation*}
    \begin{split}
      &\{a_0,b_0,c_0,a_1,b_1,c_1\}, \{a_0,b_0,c_0,a_1,b_1,a_2,b_2\}, \{a_0,b_0,c_0,a_1,c_1,a_2,c_2\},\\
      &\{a_0,b_0,c_0,b_1,c_1,b_2,c_2\}, \{a_0,b_0,a_1,b_1,a_2,b_2,a_3,b_3\},\\
      &\{a_0,c_0,a_1,c_1,a_2,c_2,a_3,c_3\}, \{b_0,c_0,b_1,c_1,b_2,c_2,b_3,c_3\}.
    \end{split}
  \end{equation*}
\end{example}

The next two results show how to construct certain minimal vertex
covers of jet graphs. Both results are illustrated by Example
\ref{exa:4}. The same example also shows not all minimal vertex covers
of jet graphs arise from the next two results.

\begin{proposition}\label{pro:3}
  Let $G$ be a graph with vertex set $\{x_1,\dots,x_n\}$.
  \begin{enumerate}
  \item For every $s\in \mathbb{N}$, the set
    \begin{equation*}
      \{x_{1,t},\dots,x_{n,t} \,|\, t=0,\dots,s\}
    \end{equation*}
    is a minimal vertex cover of $\mathcal{J}_{2s+1} (G)$.
  \item If $\{x_{i_1},\dots,x_{i_c}\}$ is a minimal vertex cover of
    $G$, then, for every $s\in \mathbb{N}$, the set
    \begin{equation*}
      \{x_{1,t},\dots,x_{n,t} \,|\, t=0,\dots,s-1\}
      \cup \{x_{i_1,s},\dots,x_{i_c,s}\}
    \end{equation*}
    is a minimal vertex cover of $\mathcal{J}_{2s} (G)$.
  \end{enumerate}
\end{proposition}

\begin{proof}
  \begin{enumerate}[wide]
  \item\label{item:10} Fix $s\in\mathbb{N}$. By Lemma \ref{lem:1}, an edge of
    $\mathcal{J}_{2s+1} (G)$ has the form $\{x_{j,u}, x_{k,v}\}$ where
    $\{x_j,x_k\}$ is an edge of $G$ and $u+v\leqslant 2s+1$. It
    follows that either $u\leqslant s$ or $v\leqslant s$. Hence either
    $x_{j,u}$ or $x_{k,v}$ belongs to the set
    \begin{equation*}
      C = \{x_{1,t},\dots,x_{n,t} \,|\, t=0,\dots,s\}.
    \end{equation*}
    This shows $C$ is a vertex cover of $\mathcal{J}_{2s+1} (G)$.

    Next we show $C$ is a minimal vertex cover.  Fix indices
    $i\in\{1,\dots,n\}$ and $t\in\{0,\dots,s\}$. We show that
    $C\setminus \{x_{i,t}\}$ is not a vertex cover. Since $G$ is
    connected, there is an edge $\{x_i,x_j\}$ in $G$ for some
    $j\in \{1,\dots,n\}$ with $j\neq i$. Then the set
    $\{x_{i,t},x_{j,2s+1-t}\}$ is an edge of $\mathcal{J}_{2s+1} (G)$
    and neither of its endpoints belongs to $C\setminus \{x_{i,t}\}$
    because $t\leqslant s$ implies $2s+1-t\geqslant s+1$. Therefore
    $C$ is a minimal vertex cover.
  \item Fix $s\in\mathbb{N}$. By Lemma \ref{lem:1}, an edge of
    $\mathcal{J}_{2s} (G)$ has the form $\{x_{j,u}, x_{k,v}\}$ where
    $\{x_j,x_k\}$ is an edge of $G$ and $u+v\leqslant 2s$. If either
    $u\leqslant s-1$ or $v\leqslant s-1$, the same argument used in
    part \ref{item:10} shows that $\{x_{j,u}, x_{k,v}\}$ has an
    endpoint belonging to the set
    \begin{equation*}
      C = \{x_{1,t},\dots,x_{n,t} \,|\, t=0,\dots,s-1\}
      \cup \{x_{i_1,s},\dots,x_{i_c,s}\}.
    \end{equation*}
    Otherwise, $u=v=s$. In this case, we can use the fact that
    $\{x_{i_1},\dots,x_{i_c}\}$ is a minimal vertex cover of $G$ to
    deduce that either $j$ or $k$ is in $\{i_1,\dots,i_c\}$. It
    follows that either $x_{j,u}$ or $x_{k,v}$ is in
    $\{x_{i_1,s},\dots,x_{i_c,s}\}$, and hence in $C$.  This shows $C$
    is a vertex cover of $\mathcal{J}_{2s} (G)$.

    Next we show $C$ is a minimal vertex cover.  The same argument
    used in part \ref{item:10} shows that, for $i\in\{1,\dots,n\}$ and
    $t\in\{0,\dots,s-1\}$, $C\setminus \{x_{i,t}\}$ is not a vertex
    cover. Now fix an index $p\in\{1,\dots,c\}$. We show that
    $C\setminus \{x_{i_p,s}\}$ is not a vertex cover.  Since
    $\{x_{i_1},\dots,x_{i_c}\}$ is a minimal vertex cover of $G$,
    there is an edge $\{x_j,x_k\}$ whose endpoints are not contained
    in the set $\{x_{i_1},\dots,x_{i_c}\} \setminus \{x_{i_p}\}$. Note
    however that one endpoint of $\{x_j,x_k\}$ is contained in
    $\{x_{i_1},\dots,x_{i_c}\}$. If we assume, without loss of
    generality, that endpoint is $x_j$, then we have $j=i_p$ and
    $k\notin \{i_1,\dots,i_c\}$. Now $\{x_{j,s},x_{k,s}\}$ is an edge
    of $\mathcal{J}_{2s} (G)$ whose endpoints are not contained in
    $C\setminus \{x_{i_p,s}\}$. Therefore $C$ is a minimal vertex
    cover.
  \end{enumerate}
\end{proof}

\begin{proposition}\label{pro:4}
  Let $G$ be a graph with vertex set $\{x_1,\dots,x_n\}$. If
  $\{x_{i_1},\dots,x_{i_c}\}$ is a minimal vertex cover of $G$, then,
  for every $s\in \mathbb{N}$, the set
  \begin{equation*}
    \{x_{i_1,t},\dots,x_{i_c,t} \,|\, t=0,\dots,s\}
  \end{equation*}
  is a minimal vertex cover of $\mathcal{J}_s (G)$.
\end{proposition}

\begin{proof}
  Fix $s\in\mathbb{N}$. By Lemma \ref{lem:1}, an edge of
  $\mathcal{J}_s (G)$ has the form $\{x_{j,u}, x_{k,v}\}$ where
  $\{x_j,x_k\}$ is an edge of $G$ and $u+v\leqslant s$. Since
  $\{x_{i_1},\dots,x_{i_c}\}$ is a vertex cover of $G$, there is an
  index $p\in\{1,\dots,c\}$ such that either $j=i_p$ or $k=i_p$. Then
  $\{x_{j,u}, x_{k,v}\}$ has an endpoint in the set
  \begin{equation*}
    C = \{x_{i_1,t},\dots,x_{i_c,t} \,|\, t=0,\dots,s\}.
  \end{equation*}
  This shows $C$ is a vertex cover of $\mathcal{J}_s (G)$.

  Next we show $C$ is a minimal vertex cover.  Fix indices
  $p\in\{1,\dots,c\}$ and $t\in\{0,\dots,s\}$. We show that
  $C\setminus \{x_{{i_p},t}\}$ is not a vertex cover.  Since
  $\{x_{i_1},\dots,x_{i_c}\}$ is a minimal vertex cover of $G$, there
  is an edge $\{x_j,x_k\}$ whose endpoints are not contained in the
  set $\{x_{i_1},\dots,x_{i_c}\} \setminus \{x_{i_p}\}$. Note however
  that one endpoint of $\{x_j,x_k\}$ is contained in
  $\{x_{i_1},\dots,x_{i_c}\}$. If we assume, without loss of
  generality, that endpoint is $x_j$, then we have $j=i_p$ and
  $k\notin \{i_1,\dots,i_c\}$. Now $\{x_{j,t},x_{k,0}\}$ is an edge of
  $\mathcal{J}_s (G)$ whose endpoints are not contained in
  $C\setminus \{x_{{i_p},t}\}$. Therefore $C$ is a minimal vertex
  cover.
\end{proof}

Recall that a graph is \emph{well covered} if all its minimal vertex
covers have the same cardinality \cite[\S 1]{MR289347}. From the
algebraic perspective, this means the edge ideal of the graph is
equidimensional, i.e., all its irreducible components have the same
dimension. As complete graphs are well covered, Example \ref{exa:4}
shows that the jets of a well covered graph need not be well
covered. In fact, applying Propositions \ref{pro:3} and \ref{pro:4}
often allows one to show the jets of a graph are not well
covered. However, consider the following example.


\begin{example}
  Consider the 4-cycle $C_4 = (a c b d)$. Figures \ref{fig:3},
  \ref{fig:4}, and \ref{fig:5} show the cycle along with its first and
  second jets. The minimal vertex covers of $C_4$ are $\{a,b\}$, and
  $\{c,d\}$; in particular, $C_4$ is well covered. Using the Macaulay2
  package \texttt{EdgeIdeals}, we found the following information. The
  minimal vertex covers of $\mathcal{J}_1 (C_4)$ are:
  \begin{equation*}
    \{a_0,b_0,c_0,d_0\}, \{a_0,b_0,a_1,b_1\}, \{c_0,d_0,c_1,d_1\}.
  \end{equation*}
  The minimal vertex covers of $\mathcal{J}_2 (C_4)$ are:
  \begin{equation*}
    \begin{split}
      &\{a_0,b_0,c_0,d_0,a_1,b_1\}, \{a_0,b_0,c_0,d_0,c_1,d_1\},\\
      &\{a_0,b_0,a_1,b_1,a_2,b_2\}, \{c_0,d_0,c_1,d_1,c_2,d_2\}.
    \end{split}
  \end{equation*}
  This data shows that the jet graphs $\mathcal{J}_1 (C_4)$
  and $\mathcal{J}_2 (C_4)$ are well covered.
\end{example}

Recall that a \emph{very well covered} graph is one whose minimal
vertex covers all contain half the vertices of the graph
\cite{MR677051}. For example, the complete bipartite graphs $K_{n,n}$
are very well covered, since their minimal vertex covers are precisely
the two parts in the bipartition of the vertex set. Note that the
4-cycle is isomorphic to the complete bipartite graph
$K_{2,2}$. Inspired by the previous example, we prove the following
result.

\begin{theorem}\label{thm:4}
  For all integers $n\geqslant 1$ and $s\geqslant 0$, the graph
  $\mathcal{J}_s (K_{n,n})$ is very well covered.
\end{theorem}

\begin{proof}
  Denote $x_1,\dots,x_n,y_1,\dots,y_n$ the vertices of $K_{n,n}$ so
  that the edges of $K_{n,n}$ are of the form $\{x_i,y_j\}$ for
  $i,j\in \{1,\dots,n\}$. For $s\in \mathbb{N}$, the graph
  $\mathcal{J}_s (K_{n,n})$ has vertex set:
  \begin{equation*}
    \{x_{1,t},\dots,x_{n,t},y_{1,t},\dots,y_{n,t} \,|\, t=0,\dots,s\}.
  \end{equation*}
  In particular, $\mathcal{J}_s (K_{n,n})$ has $2n(s+1)$ vertices.  By
  Lemma \ref{lem:1}, the edges of $\mathcal{J}_s (K_{n,n})$ are
  precisely the sets $\{x_{i,t}, y_{j,u}\}$ with $t+u\leqslant s$.
  
  For $p=0,\dots,s+1$, define the set
  \begin{equation*}
    C_p = \{ x_{1,t},\dots,x_{n,t} \,|\, t=0,\dots,s-p\}
    \cup \{ y_{1,u},\dots,y_{n,u} \,|\, u=0,\dots,p-1\}.
  \end{equation*}
  We claim that the sets $C_p$ for $p=0,\dots,s+1$ are all the minimal
  vertex covers of $\mathcal{J}_s (K_{n,n})$. Note that the sets $C_p$
  all have cardinality $n(s-p+1) + np = n(s+1)$. Therefore, once the
  claim is proven, it will follow that $\mathcal{J}_s (K_{n,n})$ is
  very well covered.

  We start by showing each set $C_p$ is a vertex cover of
  $\mathcal{J}_s (K_{n,n})$. Fix the integer $p\in
  \{0,\dots,s+1\}$. An edge $\{x_{i,t},y_{j,u}\}$ of
  $\mathcal{J}_s (K_{n,n})$ has $t+u\leqslant s$. If $t\leqslant s-p$,
  then $x_{i,t} \in C_p$. If $t> s-p$, then we have
  $t+u\leqslant s<t+p$ so $u<p$ and $y_{j,u} \in C_p$. Either way, the
  edge $\{x_{i,t},y_{j,u}\}$ has a vertex in $C_p$, therefore $C_p$ is
  a vertex cover of $\mathcal{J}_s (K_{n,n})$.

  Next we show each set $C_p$ is a minimal vertex cover of
  $\mathcal{J}_s (K_{n,n})$.  Fix the integer $p\in \{1,\dots,s+1\}$
  and let $q \in \{0,\dots,p-1\}$. We show that, for every
  $j\in\{1,\dots,n\}$, the set $C_p \setminus \{y_{j,q}\}$ is not a
  vertex cover of $\mathcal{J}_s (K_{n,n})$. In fact, for every
  $i\in\{1,\dots,n\}$, the edge $\{x_{i,s-q}, y_{j,q}\}$ has no
  endpoint in $C_p \setminus \{y_{j,q}\}$ because $q\leqslant p-1 < p$
  implies $s-p<s-q$. The proof when a vertex $x_{i,q}$ is removed from
  $C_0$ is similar. Hence each $C_p$ is a minimal vertex cover of
  $\mathcal{J}_s (K_{n,n})$.

  Finally, we show that every minimal vertex cover of
  $\mathcal{J}_s (K_{n,n})$ is equal to $C_p$ for some
  $p=0,\dots,s+1$. Let $C$ be a minimal vertex cover of
  $\mathcal{J}_s (K_{n,n})$. If
  $C \supseteq \{ y_{1,u},\dots,y_{n,u} \,|\, u=0,\dots,s\}$, then
  $C= C_{s+1}$ by minimality. Otherwise, let
  \begin{equation*}
    p = \min \{ u\in \mathbb{N} \,|\,
    y_{j,u} \notin C \text{ for some } j=1,\dots,n\}.
  \end{equation*}
  It follows from this definition that, for every $j\in\{1,\dots,n\}$
  and every $u\in\{0,\dots,p-1\}$, $y_{j,u} \in C$. We will show that
  $C = C_p$. It is enough to show that $x_{i,t} \in C$ for every
  $i\in\{1,\dots,n\}$ and every $t\in\{0,\dots,s-p\}$. Let
  $t\in \{0,\dots,s-p\}$. Then, for every $i,j\in\{1,\dots,n\}$,
  $\{x_{i,t}, y_{j,p}\}$ is an edge of $\mathcal{J}_s (K_{n,n})$
  because $t+p\leqslant s-p+p=s$. However, by definition of $p$, there
  exists $j\in\{1,\dots,n\}$ such that $y_{j,p} \notin C$. Since $C$
  is a vertex cover, we conclude that $x_{i,t} \in C$ for every
  $i\in\{1,\dots,n\}$. Therefore $C \supseteq C_p$. By minimality, we
  conclude $C=C_p$.

  Thus our claim is proven and $\mathcal{J}_s (K_{n,n})$ is very well
  covered.
\end{proof}

\section{Open Questions}
\label{sec:open-questions}

We collect in this section a few unanswered questions that attracted
our interest while working on this project.

In \S \ref{sec:when-are-jet}, we established that jets of graphs with
diameter greater than or equal to 3 are not co-chordal (with the
possible exception of the zero jets). We also showed that jets of
graphs of diameter 1 are co-chordal (cf. Proposition \ref{pro:1}),
while the situation can be varied for graphs of diameter 2. The
following remains open.

\begin{problem}
  Characterize the graphs $G$ of diameter 2 such that for all
  $s\in \mathbb{N}$ the jets $\mathcal{J}_s (G)$ are co-chordal.
\end{problem}

Example \ref{exa:3} contains an example of a co-chordal graph of
diameter 2 whose first jets are co-chordal and whose second jets are
not co-chordal. This raises the following question.

\begin{problem}
  Given an arbitrary positive integer $n$, is there a co-chordal graph
  $G$ of diameter 2 such that for every positive integer
  $i\leqslant n$ the graph $\mathcal{J}_i (G)$ is co-chordal while the
  graph $\mathcal{J}_{n+1} (G)$ is not co-chordal?
\end{problem}

In \S \ref{sec:vertex-covers-jet}, we illustrated how certain minimal
vertex covers of jet graphs are related to the minimal vertex covers
of the original graph (cf. Proposition \ref{pro:3} and
\ref{pro:4}). We also noticed that our constructions do not capture
all minimal vertex covers of jet graphs (cf. Example
\ref{exa:4}). This leads us to the following question.

\begin{problem}
  Given a graph $G$ and $s\in \mathbb{N}$, is it possible to describe
  all minimal vertex covers of $\mathcal{J}_s (G)$ in terms of the
  minimal vertex covers of $G$? If not for an arbitrary graph $G$, can
  this be achieved for certain families of graphs?
\end{problem}

Even if minimal vertex covers of a jet graph cannot be easily
described, it would still be interesting to know their number, as this
is also the number of irreducible components in the decomposition of
the edge ideal of the jet graph.

\begin{problem}
  Given a graph $G$ and $s\in \mathbb{N}$, how many minimal vertex
  covers does $\mathcal{J}_s (G)$ have and how is this number related
  to the number of minimal vertex covers of $G$? What is the
  asymptotic behavior of the number of minimal vertex covers of
  $\mathcal{J}_s (G)$ as $s$ goes to infinity? If not for an arbitrary
  graph $G$, can either question be answered for certain families of
  graphs?
\end{problem}

We concluded \S \ref{sec:vertex-covers-jet} by observing that jets of
well covered graphs are not necessarily well covered. However, we also
showed that the jets of the complete bipartite graphs $K_{n,n}$ are
very well covered (cf. Theorem \ref{thm:4}).  The complete bipartite
graphs $K_{n,n}$ are not the only very well covered graphs. For
example, as observed by Favaron \cite[\S 0]{MR677051}, the graph in
Figure \ref{fig:10} is very well covered.
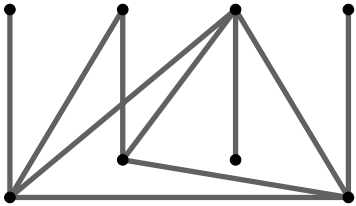
\begin{figure}[htb]
  \definecolor{wrwrwr}{rgb}{0.3803921568627451,0.3803921568627451,0.3803921568627451}
  \begin{tikzpicture}[line cap=round,line join=round,>=triangle 45,x=1cm,y=1cm,scale=.5]
    \draw [line width=2pt,color=wrwrwr] (0,5)--(0,0)--(9,0)--(9,5);
    \draw [line width=2pt,color=wrwrwr] (3,5)--(0,0)--(6,5)--(3,1)--(3,5);
    \draw [line width=2pt,color=wrwrwr] (3,1)--(9,0)--(6,5)--(6,1);
    \begin{scriptsize}
      \draw [fill=black] (0,0) circle (4pt);
      \draw [fill=black] (9,0) circle (4pt);
      \draw [fill=black] (3,1) circle (4pt);
      \draw [fill=black] (6,1) circle (4pt);
      \draw [fill=black] (0,5) circle (4pt);
      \draw [fill=black] (3,5) circle (4pt);
      \draw [fill=black] (6,5) circle (4pt);
      \draw [fill=black] (9,5) circle (4pt);
    \end{scriptsize}
  \end{tikzpicture}
  \caption{Favaron's very well covered graph}
  \label{fig:10}
\end{figure}
With the help of Macaulay2 and the package \texttt{EdgeIdeals}, we
found that the jets of this graph up to the third order are all very
well covered. In light of these computations and the result of Theorem
\ref{thm:4}, we make the following conjecture.

\begin{conjecture}
  If $G$ is a very well covered graph, then for all $s\in \mathbb{N}$
  the graph $\mathcal{J}_s (G)$ is very well covered.
\end{conjecture}

Due to the limited amount of time available for this project, we chose
to focus on graph theoretic properties that play a significant role in
combinatorial commutative algebra. However, there are many other graph
theoretic properties and invariants that we did not consider. Some of
these properties and invariants could merit further investigation in
relation to jets of graphs. For example, are there jet graphs that are
Eulerian or Hamiltonian? When are jet graphs planar? How are the
diameter, matching number, and Tutte polynomial of jet graphs related
to the same invariants of the original graphs?

\begin{problem}
  \begin{itemize}
  \item Consider a graph theoretic property. Is it preserved by taking
    jets?  If so, to what order?
  \item Consider a graph theoretic invariant. How is this invariant of
    the $s$-jets of a graph related to the same invariant of the
    original graph?
  \end{itemize}
\end{problem}

Another aspect we did not have time to consider is how this work could
generalize beyond graphs. The paper of Goward and Smith applies to
arbitrary monomial ideals, not only quadratic ideals. As such, our
definition of graphs of $s$-jets from \S \ref{sec:backgr-defin} can be
extended to hypergraphs \cite{MR1013569}.

\begin{problem}
  Extend the definition of jets to hypergraphs and study their
  properties. For example, do jets preserve the chromatic number of a
  hypergraph? How are vertex covers of jet hypergraphs related to
  vertex covers of the original hypergraph?
\end{problem}

In addition, the radical of the jets of a squarefree monomial ideal is
again a monomial ideal by Goward and Smith. Another avenue for
research is the study of these jet ideals using Stanley-Reisner theory
\cite{MR2110098}.

Our project being motivated by combinatorial commutative algebra, we
limited ourselves to undirected graphs. It is however possible to
consider constructions similar to our jets of graphs in the context of
directed graphs. The following definition was suggested to us by
Dr.~Robert Short and builds upon Lemma \ref{lem:1}.
\begin{definition}
  Let $G = (V,A)$ be a simple directed graph with vertex set $V$ and
  directed edge set $A$. The \emph{graph of $s$-jets} of $G$ is the
  directed graph with
  \begin{itemize}
  \item vertex set $\{v_0,v_1,\dots,v_s \,|\, v\in V\}$, and
  \item directed edge set
    $\{ (v_i, w_j) \,|\, (v,w)\in A, i+j\leqslant s\}$.
  \end{itemize}
\end{definition}
It may be interesting to consider properties of jets of directed
graphs and explore possible connections with algebra.

\newcommand{\etalchar}[1]{$^{#1}$}

\end{document}